\documentclass[12pt]{amsart}
\usepackage{amscd}
\usepackage{amsmath, amsfonts,amssymb,amsthm,mathrsfs,xypic}
\usepackage{graphicx}
\xyoption{curve}

\usepackage{fancybox}
\usepackage[all]{xy}
\numberwithin{equation}{section}

\theoremstyle{plain}
\newtheorem{theorem}[equation]{Theorem}

\newtheorem{lemma}[equation]{Lemma}

\theoremstyle{definition}
\newtheorem{definition}[equation]{Definition}

\theoremstyle{remark}
\newtheorem{remark}[equation]{Remark}

\newcommand{\ul}{\underline}
\newcommand{\cok}{\mathrm{cok}}

\def\C{\mathcal{C}}
\def\M{\mathcal{M}}

\def\H{\mathcal{H}}

\def\E{\mathcal{E}}
\def\I{\mathcal{I}}

\def\A{\mathcal{A}}
\def\F{\mathcal{F}}
\def\W{\mathcal{W}}

\begin{document}

\title [\tiny{A note on model structures on arbitrary Frobenius categories}]{A note on model structures on arbitrary Frobenius categories}
\author [\tiny{Zhi-Wei Li}] {Zhi-Wei Li}

\date{\today}
\thanks{The author was supported by National Natural Science Foundation
of China (No.s 11671174 and 11571329).}

\date{\today}


\maketitle
\begin{center}
{\tiny{{\it Zhi-Wei Li}, School of Mathematics and Statistics, Jiangsu Normal University, \\ 101 Shanghai Road,
Xuzhou 221116, Jiangsu, P. R. China, e-mail:zhiweili@jsnu.edu.cn}}
\end{center}
\begin{abstract}  We show that there is a model structure in the sense of Quillen on an arbitrary Frobenius category $\F$ such that the homotopy category of this model structure is equivalent to the stable category $\underline{\F}$ as triangulated categories. This seems to be well-accepted by experts but we were unable to find a complete proof for which in the literature. When $\F$ is a weakly idempotent complete (i.e. every split monomorphism is an inflation ) Frobenius catgory, the model structure we constructed is an exact (closed) model structure in the sense of \cite[J. Pure Appl. Algebra, 215(2011), 2892-2902]{Gillespie11}.

\end{abstract}

\setcounter{tocdepth}{1}

\section{Introduction}
It is well known that stable categories of Frobenius categories \cite{Happel88}, as well as certain homotopy categories of Quillen model structures \cite{Quillen67}, are two important methods for constructing triangulated categories. This note is aimed to make it clear that the first method can always be recovered from the second.

Recall that an {\it exact category} in the sense of Quillen \cite{Quillen73} is a pair $(\A, \E)$ in which $\A$ is an additive category and $\E$ is a class of {\it kernel-cokernel sequences} in $\A$. Recall that a sequence $A\stackrel{i}\to B\stackrel{d}\to C $ is a kernel-cokernel sequence if $i=\ker d$ and $d=\cok i$. There is an axiomatic description of an exact category in \cite[Appendix A]{Keller90}. Following \cite{Keller90}, in a kernel-cokernel sequence $X\xrightarrow{f} Y\xrightarrow{g} Z$ in $\E$, the morphism $f$ is called an {\it inflation}, $g$ is called a {\it deflation} and the short exact sequence itself is called a {\it conflation}. We will use $\mathrm{ Inf}(\F)$ and $\mathrm{ Def}(\F)$ to denote the class of inflations and deflations of $\F$ respectively.

 An exact category $(\F, \E)$ is called a {\it Frobenius} category if $\F$ has enough projective objects (relative to $\E$) and injective objects (relative to $\E$), and the projective objects coincide with the injective ones. We use $\I$ to denote the subcategory of injective-projective objects of $\F$. Note that $\I$ is closed under direct summands by \cite[Corollary 11.6]{Buhler10}. Recall that given two morphisms $f, g\colon X\to Y$ in $\F$, $f$ is said to be {\it stably equivalent} to $g$, written $f\sim g$, if $f-g$ factors through some object in $\I$.  Denote by $\underline{\F}$ the stable category of $\F$ whose objects are objects in $\F$ and whose morphisms are stable equivalence classes of morphisms in $\F$. It has a well known triangulated structure as shown in \cite[Theorem 2.6]{Happel88}.

In a Frobenius category $\F$, a morphism $f$ is called a {\it stable equivalence} if it is an isomorphism in the stable category $\ul{\F}$. We define the following three classes of morphisms in $\F$:
\begin{equation*}  \C of(\F)=\mathrm{ Inf}(\F), \ \F ib(\F)=\mathrm{ Def}(\F), \ \W e(\F)=\{stable\ equivalences\}
\end{equation*}
The following result shows that $\M_\F=(\C of(\F), \F ib(\F), \W e(\F))$ is a classical model structure (that is, not necessarily  a {\it closed model structure}) on $\F$ in the sense of \cite[Section I.1, Page 1.1, Definition 1]{Quillen67} (see Definition \ref{defn:model} for details). This is inspired by \cite[Theorem 2.2.12]{Hovey99}, \cite[Theorem 2.6]{Hovey02}, \cite[Corollary 3.4]{Gillespie11} and \cite[Proposition 4.1]{Gillespie13} for the cases when $\F$ is the module category of a Frobenius ring and a weakly idempotent complete (i.e. every split monomorphism is an inflation) exact category. More important, this result shows also that the associated homotopy category $\H o(\M_\F)$ is equivalent to the stable category $\ul{\F}$ preserving the triangulated structures constructed by Quillen in  of \cite[Section I.2, Page 2.9, Theorem 2]{Quillen67} and Happel in \cite[Theorem 2.6]{Happel88}.
\begin{theorem}\ \label{mainthm}
 Let $\F$ be a Frobenius category. Then $\M_\F$ is a classical model structure on $\F$ and the associated homotopy category $\H o(\M_\F)$ is equivalent to the stable category $\underline{\F}$ as triangulated categories.
\end{theorem}

\section{The proof of Theorem \ref{mainthm}}

\subsection*{The classical model structures}
We recall the definition of a model structure in the sense of \cite[Section I.1, Page 1.1, Definition 1]{Quillen67}, which is called a {\it classical model structure} here. The reason is that the modern definition of a model structure often corresponds to what Quillen called a {\it closed model structure} \cite{DS95, Hovey99, Hirschhorn03}.

\begin{definition} \label{defn:model}\ A {\it classical Quillen model structure} on an exact category $\A$ is three classes of morphisms of $\A$ called cofibrations, fibrations and weak equivalences, denoted by $\C of$, $\F ib$ and $\W e$ respectively. It requires the following properties.
\vskip5pt

$(\mathrm{ M}1)$ (Lifting property) \ Given a commutative diagram in $\A$:
\[
\xymatrix{
A\ar[r]^{f}\ar[d]_{i} & X\ar[d]^{p}\\
B \ar[r]_{g}\ar@{.>}[ru]^h& Y}
\]
if either $i$ is a cofibration and $p$ is a trivial fibration (i.e., a fibration which is a weak equivalence), or $i$ is a trivial cofibration (i.e., a cofibration which is a weak equivalence) and $p$ is a fibration, then there exists a morphism $h\colon B\to X$ such that $hi=f$ and $ph=g$.

\vskip5pt
$(\mathrm{ M}2)$ \ Fibrations are deflations and closed under composition and pullback. Cofibrations are inflations and closed under composition and pushout. The pullback of a trivial fibration is a weak equivalence, the pushout of a trivial cofibration is a weak equivalence.
\vskip5pt
$(\mathrm{ M}3)$ (Factorization property)  \ Any morphism $f$ in $\A$  can be factored in two ways: (i)\ $f=pi$, where $i$ is a cofibration and  $p$ is a trivial fibration, and (ii)\ $f=pi$, where $i$ is a trivial cofibration and $p$ is a fibration.
\vskip5pt
$(\mathrm{ M}4)$ (Two out of three property)\ If $f, g$ are composable morphisms in $\A$ and if two of the three morphisms $f, g$ and $gf$ are weak equivalences, so is the third.
\end{definition}

Note that by \cite[Section I.5, Page 5.5, Proposition 2]{Quillen67}, a classical Quillen model structure is closed if and only if the classes of fibrations, cofibrations, and weak equivalences are each closed under retracts.

Let $\M=(\C of, \F ib, \W e)$ be a classical model structure on an exact category $\F$. As described in the  beginning of \cite[Section 4]{Gillespie11}, we can construct Quillen's homotopy category of $\M$ without the full assumption that $\F$ has all finite limits and colimits. Recall that an object $A\in \F$ is called {\it cofibrant} if $0\to A\in \mathcal{C} of$, it is called {\it fibrant} if $A\to 0\in \mathcal{F}ib$, and it is called {\it trivial} if $0\to A\in \mathcal{W}e$. Each object of $\F$ which is both cofibrant and fibrant is called {\it bifibrant}, we use $\M_{cf}$ to denote the subcategory of $\F$ consisting of bifibrant objects. A {\it path object} for an object $A\in \F$ is an object $A^I$ of $\F$ together with a factorization of the diagonal map $A\stackrel{s}\to A^I\xrightarrow{(p_0,p_1)} A\prod A $ where $s$ is a trivial cofibration and $(p_0,p_1)$ a fibration such that $p_0s=p_1s=1_A$. Two morphisms $f,g\colon A\to B$ in $\F$ are called {\it right homotopic} if there exists a path object $B^I$ for $B$ and a morphism $H\colon A\to B^I$ such that $f=p_0H$ and $g=p_1H$. In this case, $H$ is called a {\it right homotopy} from $f$ to $g$. If $f$ and $g$ are right homotopic, we denote this by $f\stackrel{r}\sim g$. Dually, one can define the notions of {\it cylinder objects}, {\it left homotopic} $\stackrel{l}\sim$, {\it left homotopies}, respectively. In the subcategory of bifibrant objects, the relations $\stackrel{r}\sim$ and $\stackrel{l}\sim$ coincide and yield an equivalence relation $\stackrel{h}\sim$. The {homotopy category} $\H o(\M)$ of $\M$ is the Gabriel-Zisman localization \cite{Gabriel/Zisman67} of $\F$ with respect to the class of weak equivalences, it is equivalent to the quotient category $\M_{cf}/{\stackrel{h}\sim}$ by Quillen's homotopy category theorem  \cite[Section I.1, Page 1.13, Theorem 1]{Quillen67}.

\subsection*{Classical model structures on Frobenius categories}

Let $\F$ be a Frobenius category. Then we define the classes of cofibrations, fibrations and weak equivalences of $\F$ as follows:
\begin{equation*}
 \C of(\F)=\mathrm{ Inf}(\F), \ \F ib(\F)=\mathrm{ Def}(\F), \ \W e(\F)=\{stable\ equivalences\}.
 \end{equation*} The following lemma characterizes trivial cofibrations and trivial fibrations.

\begin{lemma} \label{lem:trivial} $(\mathrm{i})$ \ An inflation $i$ in $\F$ is a trivial cofibration if and only if $\cok i$ is injective.

$(\mathrm{ii})$ \ A deflation $p$ in $\F$ is a trivial fibration if and only if $\ker p$ is injective.

\end{lemma}
\begin{proof}
 We only prove the first statement since the proof of the second one is similar. If $i\colon X\to Y$ is a trivial cofibration, then $i$ has a stable inverse $j\colon Y\to X$. By definition, there is an injective object $I$ and morphisms $t\colon X\to I$ and $s\colon I\to X$ such that $1_X -ji=st$. By \cite[Proposition 2.12]{Buhler10}, the morphism $\left(\begin{smallmatrix}
i\\
t
\end{smallmatrix}\right)\colon X\to Y\oplus I$ is an inflation, so it has a cokernel $J$. Consider the following commutative diagram of conflations:
\[
\xy\xymatrixcolsep{2pc}\xymatrix@C16pt@R16pt{
X \ar[r]^{i} \ar@{=}[d] & Y \ar[r] & \cok i \\
X \ar[r]^-{\left(\begin{smallmatrix}
i\\
t
\end{smallmatrix}\right)}& Y\oplus I \ar[r]^h\ar[u]_{(1,0)} & J  \ar[u]_\eta}
\endxy\]
By the dual of  of \cite[Proposition 2.12]{Buhler10}, the right square of the above diagram is a pullback. Thus $\eta\colon J\to \mathrm{ coker}i$ has a kernel $I$, so it is a deflation by \cite[Proposition 2.15]{Buhler10}. Since the morphisms $i$ and $(1, 0)$ are stable equivalences and the stable equivalences satisfies two out of three property, we know that $\left(\begin{smallmatrix}
i\\
t
\end{smallmatrix}\right)$ is a stable equivalence, and then $h$ factors through an injective object. Since $(j, s)\left(\begin{smallmatrix}
i\\
t
\end{smallmatrix}\right)=1_X$, we know that the second row in the above diagram is split. Thus there is a morphism $m\colon J\to Y\oplus I$ such that $1_J=hm$, and then $1_J$ factors through an injective object. From which we can show that $J$ is an injective object. By the conflation
$$I \to J \xrightarrow{\eta} \cok i$$
 we know that $\cok i$ is a direct summand of $J$, and so is an injective object.

Conversely, if $i$ is an inflation with $\mathrm{ coker}i\in \I$, then it is a trivial cofibration by construction. Thus we have $\C of(\F)\cap \W e(\F)=\{ i\in \mathrm{ Inf}(\F) \ | \ \cok i\in \I\}$.
\end{proof}

\begin{lemma} \label{lem:lifting}\ $(\mathrm{i})$ \ Cofibrations has the lifting property with respect to the trivial fibrations.

 $(\mathrm{ii})$ \ Trivial cofibrations has the lifting property with respective to fibrations.
\end{lemma}

\begin{proof}  (i) \ Consider the following lifting problem
\[
\xy\xymatrixcolsep{2pc}\xymatrix@C16pt@R16pt{
 A  \ar[r]^f \ar[d]_i & X\ar[d]^{p}\\
 B \ar[r]^g& Y   }
\endxy\]
where $i$ is a cofibration and $p$ a trivial fibration. Then we have a commutative diagram of conflations
 \[\xy\xymatrixcolsep{2pc}\xymatrix@C16pt@R16pt{& A\ar[r]^i\ar[d]_f & B\ar[d]^{g}\ar[r]^d& \cok i \\
\ker p \ar[r]^c& X \ar[r]^p & Y & }
\endxy\]
By Lemma \ref{lem:trivial}, $\ker p \in \I$, thus the second row in the above diagram splits. So there is a morphism $\lambda\colon Y\to X$ such that $p \lambda=1_Y$. Then $p(f-\lambda gi)=pf-p\lambda gi=pf-gi=0$ and therefore there is a morphism $v\colon A\to \ker p$ such that $f-\lambda gi=cv$. For the morphism $v$, since $\ker p$ is in $\I$, there is a morphism $u\colon B\to \ker p$ such that $v=ui$. Thus we have $f=(\lambda g+cu)i$. Let $h=\lambda g+cu$. Then $h$ is the desired lift.
\end{proof}

\begin{lemma} \label{lem:factorization}\ Every morphism $f$ in $\F$ can be factorized as $f=pi=qj$, where $p\in \F ib(\F)$, $i\in \C of(\F)\cap \W e(\F)$, $q\in \F ib(\F)\cap \W e(\F)$, $j\in \C of(\F)$.
\end{lemma}

\begin{proof} \ If $f:X\to Y$ is an inflation, take $q=\mathrm{ Id}_Y, j=f$. For the construction of $p$ and $i$, note that since $\F$ has enough projective objects, there is a projective object $P$ and a deflation $P\to \cok f$. Then we have the following pullback diagram
\[
\xymatrix{
X \ar[r]^{i} \ar@{=}[d] & Y' \ar[r] \ar[d]_p & P\ar[d]\\
X \ar[r]_{f}& Y\ar[r] & \cok f  }
\]
Since deflations are closed under taking pullbacks, we know that $p\in \F ib(\F)$ and $i\in \C of(\F)$. Since $P\in \I$, the morphism $i$ is also a weak equivalence by Lemma \ref{lem:trivial}. This gives the factorization of $f=pi$.

Dually, we can prove that the claim holds for deflations. Now suppose that $f\colon X\to Y$ is an arbitrary morphism in $\F$. It can be factorized as
$$X\xrightarrow{\left(\begin{smallmatrix}
\mathrm{ Id}_X\\
0
\end{smallmatrix}\right)} X\oplus Y\xrightarrow{(f, \mathrm{ Id}_Y)} Y$$
where $\left(\begin{smallmatrix}
\mathrm{ Id}_X\\
0
\end{smallmatrix}\right)$ is an inflation by Lemma 2.7 of \cite{Buhler10} and $(f, \mathrm{ Id}_Y)$ is a deflation by the dual of  \cite[Proposition 1.12]{Buhler10}. Write $\left(\begin{smallmatrix}
\mathrm{ Id}_X\\
0
\end{smallmatrix}\right)=p'i$ with $p'\in \F ib(\F)$ and $i\in \C of(\F)\cap \W e(\F)$. Then $p=(f, \mathrm{ Id}_Y)p'$ is in $\F ib(\F)$ satisfies $pi=(f, \mathrm{ Id}_Y)p'i=(f, \mathrm{ Id}_Y)\left(\begin{smallmatrix}
\mathrm{ Id}_X\\
0
\end{smallmatrix}\right)=f$. Similarly, write $(f, \mathrm{ Id}_Y)=q j'$ with $q\in \F ib(\F)\cap \W e(\F)$ and $j'\in \C of(\F)$. Let $j=j'\left(\begin{smallmatrix}
\mathrm{ Id}_X\\
0
\end{smallmatrix}\right)$. Then $f=qj$ with $q\in \F ib(\F)\cap \W e(\F)$ and $j\in \C of(\F)$.
\end{proof}

\subsection{The proof of Theorem \ref{mainthm}}\  We fist prove that $\M_\F$ is a classical model structure on $\F$ by verifying the axioms $(\mathrm{ M}0)-(\mathrm{ M}3)$ of Definition \ref{defn:model} one by one. By construction, all the three classes of morphisms $\C of(\F), \F ib(\F)$ and $\W e(\F)$ contain isomorphisms. By Lemma \ref{lem:lifting}, we have $(\mathrm{ M}0)$. $(\mathrm{ M}2)$ follows from Lemma \ref{lem:factorization}. Since stable equivalence satisfies two out of three property, we know that $(\mathrm{ M}3)$ holds. For the proof of $(\mathrm{ M}1)$, note that fibrations are deflations, which are stable under composition and pullback by the axiomatic description of an exact category.  Similarly, we know that cofibrations are stable under composition and pushout. Suppose that we have a pullback diagram
\[
\xymatrix{
 Z' \ar[r] \ar[d]_q & X\ar[d]^p\\
 Z\ar[r] & Y  }
\]
with $p\in \F ib(\F)\cap \W e(\F)$. By Lemma \ref{lem:trivial}, $\ker p\in \I$. By the dual of \cite[ Proposition 2.12]{Buhler10}, $q$ is a deflation with $\ker q=\ker p\in \I$. Thus by Lemma \ref{lem:trivial} again, $q\in \F ib(\F)\cap \W e(\F)$ and in particular it is a weak equivalence. Similarly, we can prove that the pushout of a morphism which is both a cofibration and a weak equivalence is a weak equivalence. Thus $\M_\F$ is a model structure on $\F$.

 With this model structure, every object in $\F$ is bifibrant. For each object $X$ in $\F$, we can choose a conflation $\Omega(X)\stackrel{\iota_X}\rightarrowtail P(X)\stackrel{p_X}\twoheadrightarrow X$ with $P(X)\in \I$ since $\F$ is a Frobenius category. Then $X\oplus P(X)$ is a path object of $X$:
  $$X\xrightarrow{\left(\begin{smallmatrix}
1_X\\
0
\end{smallmatrix}\right)} X\oplus P(X)\xrightarrow{\left(\begin{smallmatrix}
1_X & p_X\\
1_X & 0
\end{smallmatrix}\right)} X\oplus X.$$
  It is straightforward to verify that two morphisms $f, g: X\to Y$ are homotopic if and only if they are stably equivalent. Thus $\H o(\M_\F)$ is equivalent to $\underline{\F}$ by Quillen's homotopy category theorem \cite[Section I.1, Page 1.13, Theorem 1]{Quillen67}. Now we have to prove that the triangle structure on $\ul{\F}$ constructed by Quillen in \cite[Section I.2, Page 2.9, Theorem 2]{Quillen67} coincides with the one constructed by Happel in \cite[Theorem 2.6]{Happel88}. In fact, for each morphism $f:X\to Y$, since $P(X)\in \I$, there is a morphism $x_f$ such that $fp_X=p_Yx_f$, and then there exists a commutative diagram of conflations:
\[\xy\xymatrixcolsep{2pc}\xymatrix@C36pt@R36pt{ \Omega(X)\ar[r]^-{\left(\begin{smallmatrix}
\iota_X\\
0
\end{smallmatrix}\right)}\ar[d]_{\kappa_f} & X\oplus P(X) \ar[r]^-{\left(\begin{smallmatrix}
1_X & p_{X}\\
1_X & 0
\end{smallmatrix}\right)}\ar[d]^{\left(\begin{smallmatrix}
f & 0\\
0 & x_f
\end{smallmatrix}\right)} & X\oplus X \ar[d]^{\left(\begin{smallmatrix}
f & 0\\
0 & f
\end{smallmatrix}\right)}\\
 \Omega(Y)\ar[r]^-{\left(\begin{smallmatrix}
\iota_Y\\
0
\end{smallmatrix}\right)}& Y\oplus P(Y)\ar[r]^-{\left(\begin{smallmatrix}
1_Y & p_{Y}\\
1_Y & 0
\end{smallmatrix}\right)}& Y\oplus Y
}\endxy\]
Recall that the loop functor on $\ul{\F}$ defined by Quillen in \cite[Section I.2, Page 2.9, Theorem 2]{Quillen67}, denoted by $\Omega$, is defined by sending $X$ to $\Omega(X)$ and $\ul{f}$ to $\ul{\kappa}_f$. It coincides with the one defined by Happel in \cite[Theorem 2.6]{Happel88}. This is an autoequivalence of $\ul{\F}$ by \cite[Proposition 2.2]{Happel88}.

Given any fibration $f: X\to Y$ in $\F$, we have a commutative diagram of conflations:
\[
\xymatrix{
 \Omega(Y)\ar[r]^{\iota_Y}\ar[d]_{\xi_f} & P(Y)\ar[d]_{\delta_f} \ar[r]^{p_Y} & Y\ar@{=}[d] \\
 \ker f\ar[r]^{\iota_f} & X \ar[r]^{f} & Y.}
\]
Since $\F$ is an additive category, $\Omega(Y)$ is a group object in $\ul{\F}$ and giving $\ker f$ a group action of $\Omega(Y)$ is equivalent to giving a morphism from $\Omega(Y)$ to $\ker f$. See also \cite[Subsection 1.1]{Becker12}. So the left triangles in $\ul{\F}$ constructed by Quillen are isomorphic to those of the form $$\Omega(X)\stackrel{-\ul{\xi}_f}\to \ker f\stackrel{\ul{\iota}_f}\to X\stackrel{\ul{f}}\to Y,$$
see \cite[Theorem 6.2.1, Remark 7.1.3]{Hovey99} for details. Then by \cite[Lemma 2.7]{Happel88}, we know that this triangle structure coincides with the one constructed by Happel.

\begin{remark}\ For a Frobenius category $\F$, by \cite[Proposition 2.4]{Gillespie11}
or \cite[Proposition 4.1]{Gillespie13}, the classical model structure $\M_\F$ is closed if and only if the underling category $\F$ is weakly idempotent complete. In this case, we get an {\it exact model structure} in the sense of \cite{Gillespie11}\end{remark}

\noindent{\bf Acknowledgements}\ The author would like to thank the referee for reading the paper carefully and for many suggestions on mathematics and English expressions.

\vskip 10pt

\end{document}